\documentclass{amsart}
\usepackage{amstext, amssymb, amsthm, amsfonts, latexsym,bbm,tikz}
\usetikzlibrary{matrix}
\usepackage{varioref}
\usepackage{chemmacros}
\usepackage{todonotes}
\usepackage{enumitem}
\usepackage{bbm}
\usepackage[utf8]{inputenc}
\usepackage[english]{babel}
\usepackage{tikz}
\usetikzlibrary{arrows} 
\usepackage{comment}
\usepackage{tikz-cd}
\usepackage{hyperref}

\usepackage{float}
\restylefloat{table}
\usepackage{placeins}

\usepackage{color}
 
\definecolor{myred}{RGB}{ 204,102,119}

\definecolor{myblue}{RGB}{68,119,170}
\definecolor{myyellow}{RGB}{221,204,119}
\definecolor{mygray}{gray}{0.6}

\newtheorem{theorem}{Theorem}[section]
\newtheorem{lemma}[theorem]{Lemma}
\newtheorem{corollary}[theorem]{Corollary}
\newtheorem{proposition}[theorem]{Proposition}
\newtheorem{claim}[theorem]{Claim}

\theoremstyle{definition}
\newtheorem{definition}[theorem]{Definition}
\newtheorem{example}[theorem]{Example}

\theoremstyle{remark}
\newtheorem{remark}[theorem]{Remark}

\theoremstyle{Conjecture/open problem}

\def\la{\lambda}

\def\a{\alpha}
\def\be{\beta}

\def\Ga{\Gamma}
\def\1{\mathbf{1}}
\def\la{\lambda}

\def\p{^{\prime}}

\def\k{\kappa}

\def\cG{\mathcal{G}}
\def\cR{\mathcal{R}}
\def\cS{\mathcal{S}}
\def\cC{\mathcal{C}}

\def\S{\mathbb{S}}

\def\R{\mathbb{R}}
\def\N{\mathbb{N}}
\def\M{\mathcal{M}}
\def\Z{\mathbb{Z}}

\numberwithin{equation}{section}

\begin{document}

\title{Stationary distributions and condensation in autocatalytic CRN}

\author[Linard Hoessly]{Linard Hoessly}
\address{(Linard Hoessly) Department of Mathematics, University of Fribourg, Chemin du Mus\'ee 23, CH-1700, Fribourg, CH.}
\email{linard.hoessly@unifr.ch}
\date{}

\author[Christian Mazza]{Christian Mazza}
\address{(Christian Mazza) Department of Mathematics, University of Fribourg, Chemin du Mus\'ee 23, CH-1700, Fribourg, CH.}
\email{christian.mazza@unifr.ch}
\date{}
\subjclass[2010]{92B05,92E20,80A30,92C42,60J28,60K35,82C20}
\keywords{Reaction networks, mass-action system, Product-form stationary distributions, Markov process, inclusion process, condensation }

\begin{abstract} 
We investigate a broad family of non weakly reversible stochastically modeled reaction networks (CRN), by looking at their steady-state distributions. Most known results on stationary distributions assume weak reversibility and zero deficiency. We first give explicitly product-form steady-state distributions for a class of non weakly reversible autocatalytic CRN of arbitrary deficiency. Examples of interest in statistical mechanics (inclusion process), life sciences and robotics (collective decision making in ant and robot swarms) are provided. The product-form nature of the steady-state then enables the study of condensation in particle systems that are generalizations of the inclusion process.
\end{abstract}

\maketitle


\section{Introduction}
Understanding the dynamics of reaction networks (CRN) is of central importance in a variety of contexts in life sciences and complex systems, including molecular and cellular systems biology, which are some of the most vital areas in bioscience. Two approaches are used to model reaction network systems,  either a deterministic or a stochastic model. The first is realized as a vector with concentrations of each molecular species as state space governed by a system of ordinary differential equations (ODE), whereas the second is described by a continuous-time Markov chain acting on discrete molecular counts of each molecular species. Typically the stochastic model is used for cases with low molecular numbers where stochasticity is essential for the proper description of the dynamics. Au contraire the deterministic model is used for cases with many molecules in each species and where it is assumed that coupled ODEs well approximate the concentrations.

The study of the dynamics of the deterministic model, mass-action kinetics in particular with complex balanced states \cite{horn4,feinberg1}, is a well-studied subject going back 
more than 100 yrs. \cite{Wegscheider,Michaelis_M}.
Understanding of such and more general ODEs from chemical reaction network theory developed to more subtle questions, like, e.g. multistationarity, persistence, etc. \cite{overv_Gorban}. Conversely, the stochastic system is analyzed via master equation. No analytic solutions are known for most systems, even concerning stationary distributions.  Consequently simulation methods and approximation schemes of different exactness, roughness and rigor were developed in order to understand such systems \cite{gillespieb,gardiner,cao,irrev},  making a systematic investigation of fundamental effects of noise and statistical inference a demanding job. 
Our results are a step towards the rigorous analysis of the product-form stationary distribution of non-weakly reversible ergodic stochastic CRNs of arbitrary deficiency. We 
exhibit product-form stationary distributions $\pi_N$ for a large class of autocatalytic mass preserving CRNs, including the models in \cite{Saito1,Saito2,Bianc,Bianc2,Swarm,condesates} (which were studied via simulation and approximations) and generalizing results of \cite{Houch,inclusion_proc}. The emerging infinite family of product-form functions in the stationary distributions (with Poisson form as a particular case) is also possibly interesting from the view of natural computation, where it extends the range of designable probability distributions of CRNs, see e.g. \cite{Cardelli2018,Soloveichik5393}. We illustrate the occurrence of such CRNs in interacting particle system theory and life sciences for collective decision making processes in ant or robot swarms.

The relation between the deterministic and the stochastic model as well as their differences are a focus of current research \cite{Bianc,Bianc2,Houch,Saito1,Saito2,Cappelletti,anderson2}. Kurtz \cite{kurtz2} linked the short term behavior of the adequately scaled continuous-time Markov chain to the dynamics of the ODE model. These results are based on the classical mean field scaling which assumes that the system is well mixed. Then the probability that a set of molecules meet in small volume is proportional to the product of the molecular concentrations $x_i/V$ where $x_i$ denotes the absolute number of molecules of type $S_i$, and where $V$ is the volume which is assumed to be large. Within this modeling framework, the orbits of the 
continuous-time Markov chain describing the stochastic CRNs converges as $V\to\infty$ towards the orbits of the mass-action ODEs. 
This convergence was also considered recently from the point of view of large deviation theory \cite{eckmann}. 
The insight of the complex balanced deterministic model was recently transferred to the stochastic model: A deterministic system is complex balanced if and only if the stochastically modeled system has product-form of Poisson type \cite{anderson2,Cappelletti}, where the parameter of the Poisson distributions are given by the stable equilibrium values of the related deterministic mass-action dynamic. 

CRNs with stochastic behavior differing from the behavior of the deterministic CRN due to molecular discreteness and stochasticity were identified. The mathematical analysis is based on approximations \cite{Saito1,Saito2,Bianc,Bianc2,Swarm,condesates} in the ergodic case, or on the analysis of absorbing states for absolute concentration robust CRNs \cite{Anderson14,Anderson16,Enciso16}.  In the ergodic case such behavior appeared in the literature as noise-induced bi-/multistability\cite{Bianc,Bianc2}, small-number effect \cite{Saito1,Saito2} or noise-induced transitions \cite{noise_ind}.  
Our setting includes examples from \cite{Bianc,Bianc2,Swarm} and some examples of \cite{Saito1,Saito2}. Hence we shed light on such instances by providing product-form stationary distributions and enabling exact analysis for the class of autocatalytic CRNs (see Definition \ref{autocatalyticCRN} and Remark \ref{abbr}). 
We inspect them asymptotically when the total number of molecules $N$ is large. Taking inspiration from previous works on particle systems \cite{Giardina2009,Giardina2010,inclusion_proc,Bianchi16}, we consider non mean field transition mechanisms where particles (or molecules) are located at the nodes of a graph. Particles located at some node $i$ (or of type $S_i$) can move to nearest neighbour nodes $j$. Within this new modeling framework, the rate at which a particle moves from site $i$ to site $j$ (or that a molecule of type $S_i$ is converted into a molecule of type $S_j$) is related to the absolute numbers $x_i$ and $x_j$ of species $S_i$ and $S_j$. While a classical mean field scaling with $V=N$ would lead to convergence of $\pi_N$ as $N\to \infty$ towards a point mass centered at the positive equilibrium of the deterministic mass-action ODE, the new scaling regime leads to the emergence of condensation: the stationary distribution $\pi_N$ of autocatalytic CRNs can under some conditions converge towards limiting probability measures with supports located on the faces of the probability simplex. In other words,  the set of molecules concentrates as $N\to\infty$ on a strict subset of the set of species. 
 We investigate the asymptotic behavior of the product-form stationary distribution $\pi_N$, putting emphasis on the cases of up to molecularity three in our model with respect to three different forms of condensation. We observe that monomolecular autocatalytic CRNs (see Definition \ref{autocatalyticCRN} and Remark \ref{abbr}) and complex balanced CRNs do not satisfy any form of condensation. We generalize a Theorem from \cite{inclusion_proc} to allow more general product-form functions and prove, for the up to bimolecular case, a weak form of condensation and a weak law of large numbers. In the threemolecular and higher case, we show that such systems exhibit the strongest form of condensation. 
 \subsection*{Acknowledgements}
The first author is supported by Swiss National Science Foundation grant number PP00P2\_179110 and thanks Jan Draisma and Manoj Gopalkrishnan for helpful dsicussions. Furthermore we thank Badal Joshi, Daniele Cappelletti and two anonymous referees for valuable feedback.


\section{Reaction networks}\label{CRN}
A \textbf{reaction network} is a triple $\cG=(\cS,\cC,\cR)$, where
 $\cS$ is the set of \textbf{species} $\cS=\{S_1,\cdots,S_n\}$, $\cC$ is the set of \textbf{complexes} and $\cR$ is the set of \textbf{reactions} $\cR=\{R_1,\cdots,R_r\}$. 
 
 Complexes are made up of linear combinations of species over $\Z_{\geq 0}$,  identified  with  vectors  in $\Z_{\geq 0}^n$. 
 Reactions consist of ordered tuples $(\nu, \nu\p)\in \cR$ with $\nu,\nu\p\in\cC$.  Such a reaction consumes the \textbf{reactant} $\nu$ and creates the \textbf{product} $\nu\p$. We will typically write such a reaction in the form $\nu\to \nu\p $. We will often write complexes $\nu\in\Z_{\geq 0}^n$ in the form
$\nu=\sum_{i=1}^n\nu_i S_i.$ Accordingly we slightly abuse notation at times for complexes by identifying $\nu$ with $\sum_{i=1}^n\nu_i S_i$. 

We usually describe a reaction network by its \textbf{reaction graph} which is the directed graph with vertices $\cC$ and edge set $\cR$. A connected component of the reaction graph of $\cG$ is termed a \textbf{linkage class}. We say $\nu \in \cC$ \textbf{reacts} to $\nu\p \in \cC$ if $\nu\to \nu\p$ is a reaction. A reaction network $\cG$ is \textbf{reversible} if $\nu\to \nu\p \in\cR$ whenever $ \nu\p\to \nu \in\cR$ (different to reversibility of stochastic processes), and it is \textbf{weakly reversible} if for any reaction $\nu\to \nu\p \in\cR$, there is a sequence of directed reactions beginning with $\nu\p$ as a source complex and ending with $\nu$ as a product complex. If it is not weakly reversible we say it is \textbf{non-weakly reversible}. The \textbf{molecularity} of a reaction $\nu\to \nu\p \in\cR$ is equal to the number of molecules in the reactant $|\nu|=\sum_i \nu_i$. Correspondingly we call such reactions unimolecular, bimolecular, three-molecular or n-molecular reactions. Alternatively we say a reaction has molecularity one, two, three or n. The \textbf{stochiometric subspace} is defined as $$\mathcal{T}={\rm span}_{\nu\to\nu\p\in\cR}\{\nu-\nu\p\}\subset \mathbb{R}^n,$$
and for $v\in \R^n$, the sets $(v+\mathcal{T})\cap\R^n_{\geq 0}$ are \textbf{stochiometric compatibility classes} of $\cG$. The following invariant has proven to be important in the study of complex balanced CRNs. The \textbf{deficiency}  of a reaction network $\cG$ is defined as
 $$\delta=|\cC|-\ell-{\rm dim}(\mathcal{T}),$$
 where $\ell$ is the number of linkage classes.

For each reaction $\nu\to \nu\p$ we consider a positive \textbf{rate constant} $\k_{\nu\to\nu\p}$; the  vector of reaction weights is defined by $\k\in\R_{>0}^\cR$
 and the CRN with rates is denoted by $(\cG,\k)$.
For examples of reaction networks see $\S$ \ref{s.swarm}
\subsection{Deterministic model\label{s.deterministicmodel}}
Here we review the main notions connected to the deterministic model.
This setting is usually termed \textit{deterministic mass-action kinetics}. The system of ODEs associated to the CRN $(\cG,\k)$ with mass-action kinetics is
$$\frac{d}{dt}x(t)=\sum_{\nu\to\nu\p\in\cR}\k_{\nu\to\nu\p}x(t)^\nu(\nu\p-\nu),$$
where for $a,b\in \R_{\geq 0}^n$ we define $a^b= \prod_{S_i\in\cS}a_i^{b_i}$ with convention $0^0=1$. The system then follows this ODE started from initial condition $x_0=x(0)\in \R^n$ and the dynamics of $x(t)\in \R^n$ models the vector of concentrations at time t. 
\begin{definition} \label{dete_bal}
A reaction network $(\cG,\k)$ with deterministic mass-action kinetics is called:
\begin{enumerate}
\item \textbf{detailed balanced} if and only if there exists a point $a\in \R^n_{>0}$ such that for all $\nu\to\nu\p\in\cR$, $\nu\p\to\nu\in\cR$
$$\k_{\nu\to\nu\p}a^\nu=\k_{\nu\p\to\nu}a^{\nu\p}.$$
\item \textbf{complex balanced} if and only if there exists a point $a\in \R^n_{>0}$ such that for all $\nu\in \cC$
$$\sum_{\nu\to\nu\p\in\cR}\k_{\nu\to\nu\p}a^\nu=\sum_{\nu\p\to\nu\in\cR}\k_{\nu\p\to\nu}a^{\nu\p}.$$
\end{enumerate}
\end{definition}
Note that if a CRN is detailed balanced or complex balanced, then it is necessarily weakly reversible. Also deficiency zero weakly reversible CRNs are complex balanced independent of the rate \cite{feinberg2}.
\subsection{Stochastic model}

Here we introduce the main notions connected to the stochastic model.
The setting we focus on is usually termed \textit{stochastic mass-action kinetics}. The progression of the species follows the law of a continuous-time Markov chain on state space $\Z_{\geq 0}^n$. The state at time $t$ is described by a vector $X(t)=x\in \Z_{\geq 0}^n$ which can change according to a reaction $\nu\to \nu\p$ by going from $x$ to $x+\nu\p-\nu$ with transition rate $\la_{\nu \to \nu\p}(x)$, corresponding to the consumption of $\nu$ and the production of $\nu\p$. The Markov process with intensity functions $\la_{\nu \to \nu\p}:\Z_{\geq 0}^n\to \R_{\geq 0}$ can then be given by
$$P(X(t+\Delta t)=x+\nu\p-\nu|X(t)=x)=\la_{\nu \to \nu\p}(x)\Delta t+ o(\Delta t).$$
Accordingly the generator $\mathcal{A}$ is given by
$$\mathcal{A}h(x)=\sum_{\nu \to \nu\p\in \cR} \la_{\nu \to \nu\p}(x)(h(x+\nu\p-\nu)-h(x)),$$
for $h:\Z^n\to \R$.
We focus on the usual choice, stochastic mass-action kinetics, where the transition intensity associated to the reaction $\nu\to\nu\p$ is 
\begin{equation}
\label{int}\la_{\nu \to \nu\p}(x)=\k_{\nu\to\nu\p}\frac{(x)!}{(x-\nu)!}\1_{x\geq \nu}\text{ (where }z!:=\prod_{i=1}^nz_i!\text{ for } z\in \Z^n_{\geq 0}).
\end{equation}
This uniform sampling scheme corresponds to the mean field situation where the system is \textit{well-stirred} in the sense that all particles move randomly and uniformly in the medium. The transition intensities with constants $\kappa_{\nu\to\nu\p}$ model the probability that such molecules meet in a volume element. The study of these models goes back to \cite{kelly}, \cite{whittle}. 
 In the following we fix a CRN $(\cG,\kappa)$ and introduce the main terminology from stochastics:
\begin{definition}[Decomposition of state space]
We say:\\
- A reaction $y\to y\p$ is \emph{active} on $x\in \Z_{\geq 0}^n$ if $x\geq y$.\\
- A state $u\in \Z_{\geq 0}^n$ is \emph{accessible} from $x\in \Z_{\geq 0}^n$ if there is a sequence of reactions $(y_i\to y\p_i)_{i\in [p]}$ such that:\\
 $x+\sum_{i=1}^j(y_i\p-y_i)\geq 0\forall j\in [p]$\\
 $x+\sum_{i=1}^p(y_i\p-y_i)=u$\\
 - A non-empty set $\Ga\subset\Z_{\geq 0}^n$ is an \emph{irreducible component} of $\cG$ if for all $x\in \Ga$ and all $u\in  \Z_{\geq 0}^n$, $u$ is accessible from x if and only if $u\in \Ga$.\\
 -$\cG$ is \emph{almost essential} if the state space is a union of irreducible components except for a finite number of states.\\
 \end{definition}

\subsection{Stationary distribution and product-form stationary distribution} \label{stat}

The stationary distribution $\pi_\Ga$ on an irreducible component $\Ga$ describes the long-term behavior of the Markov chain in the positive recurrent case. Then $\pi_\Ga$ is unique and corresponds to the limiting distribution (see \cite{Norris}). Note that on a finite irreducible component the stationary distribution always exists\footnote{A finite state irreducible CTMC is positive recurrent hence has a stationary distribution which is the limiting distribution.}.
Let $X(t)$ denote the underlying stochastic process associated to a reaction network on a finite irreducible component $\Ga$. Then given that the stochastic process $X(t)$ starts in $\Ga$, we have 
$$\lim_{t\to\infty}P(X(t)\in A)=\pi_\Ga(A),\text{for any } A\subset\Ga.$$
The stationary distribution is determined by the master equation of the underlying Markov chain:
\begin{equation} \label{master_eq1}
  \sum_{\nu\to\nu\p\in\cR} \pi(x+\nu-\nu\p)\la_{\nu \to \nu\p}(x+\nu -\nu\p)=\pi(x)\sum_{\nu\to\nu\p\in\cR}\la_{\nu \to \nu\p}(x),
  \end{equation}
for all $x\in\Ga$. Inserting the rate functions following mass-action kinetics gives:

\begin{equation} \label{rate}
\la_{\nu \to \nu\p}(x)=\k_{\nu\to\nu\p}\frac{(x)!}{(x-\nu)!}\1_{x\geq \nu},
\end{equation}

\begin{equation} \label{master_eq2}  \sum_{\nu\to\nu\p\in\cR} \pi(x+\nu-\nu\p)\k_{\nu\to\nu\p}\frac{(x-\nu\p+\nu)!}{(x-\nu\p)!}\1_{x\geq \nu\p}=\pi(x)\sum_{\nu\to\nu\p\in\cR}\k_{\nu\to\nu\p}\frac{(x)!}{(x-\nu)!}\1_{x\geq \nu}.\end{equation}
Solving equation (\ref{master_eq1}) is in general a challenging task, even for the mass-action case equation (\ref{master_eq2}) stays difficult. 
Remark that for mass conserving CRNs, the irreducible components are finite and the stationary distribution exists always.
Some stationary distributions of weakly reversible reaction networks are well-understood. Complex balanced CRNs have a nice and simple product-form stationary distribution.
 \begin{theorem}\cite[Theorem 4.1]{anderson2} Let $(\cG,\k)$ be a CRN that is complex balanced. Then for any irreducible component $\Ga$, the stochastic system has product-form stationary distribution of the form
 $$\pi(x)=M_\Ga \frac{c^x}{x!},x\in\Ga,$$
 where $c\in\R^n_{>0}$ is a point of complex balance and $M_\Ga$ is a normalizing constant.
 \end{theorem}
So each deterministic complex balanced CRN has its stochastic counterpart with product-form stationary distribution of Poisson-type.
One can prove (see, e.g., \cite{anderson2}) that, for zero deficiency CRNs, a network is complex balanced if and only if it is weakly reversible. This explains why most results on product-form distributions assume zero deficiency. We will go beyond this setting in the forthcoming Sections.
On the other hand by \cite[Theorem 5.1]{Cappelletti} any almost essential stochastic reaction network with product-form stationary distribution of Poisson-type is deterministically complex balanced. Notice that since complex balanced implies weakly reversible, these results do not apply to non-weakly reversible CRNs. 
\subsection{Reaction vector balance CRNs}\hfill\break
The notion of reversibility plays a fundamental role in Markov chain theory.
\begin{definition} A continous time Markov chain $X(t)$ with transition rates $q(x,y)$ is reversible with respect to the distribution $\pi$ if for all $x,y$ in the state space $\Ga$ we have
\begin{equation}\label{reversibility}
\pi(x) q(x,y)=\pi(y) q(y,x).
\end{equation}
\end{definition}
Note that the notions of reversible and detailed balanced for CRNs are not the same as the same terms used for Markov chains. A definition similar to detailed balanced (see Definition \ref{dete_bal}) for the stochastic model of CRNs was recently termed as reaction vector balanced \cite{JoshiCap,DetJoshi}: 
\begin{definition}\label{react_bal} Let $(\cG,\k)$ be a CRN. A stationary distribution $\pi$ on an irreducible component $\Ga\subseteq \Z_{\geq 0}^n$ is called {\it reaction vector balanced} if for every $x\in \Ga$ and every $a\in\Z^n$
\begin{equation}\label{ReacBal}
  \sum_{\nu\to\nu\p\in\cR:\nu -\nu\p= a} \pi(x+\nu-\nu\p)\la_{\nu \to \nu\p}(x+\nu -\nu\p)=\pi(x)\sum_{\nu\to\nu\p\in\cR:\nu -\nu\p= -a}\la_{\nu \to \nu\p}(x)
\end{equation}
\end{definition}
Rewriting (\ref{ReacBal}) as
$$\pi(x+a)q(x+a,x)=\pi(x)q(x,x+a),$$
for the rates
$$q(x+a,x)=\sum_{\nu\to\nu\p\in\cR:\nu -\nu\p= a}\la_{\nu \to \nu\p}(x+a),$$
we see that  $\pi$ is reaction vector balanced if and only if the Markov chain  transition rates given by $q(x+a,x)$ are reversible.
If a CRN is detailed balanced (Definition \ref{dete_bal}), many results are known. Detailed balance implies complex balanced, so the stochastic model has product-form stationary distribution of Poisson-type.  However, more is known, by \cite[Lemma 3.1, p.157]{whittle} and \cite{DetJoshi} for reversible reaction networks this is the case if and only if the corresponding stochastic model is Whittle stochastic detailed balanced, which implies its reversibility as a Markov chain. 

\subsection{Generalized balanced CRNs}\hfill\break
In $\S$ \ref{exa}, Remark \ref{re_cx_bal} and Example \ref{ex_cx_bal} we indicate how to combine complex balanced and autocatalytic CRNs. This context was not considered before hence we review notions of \cite{JoshiCap} to adapt and encapsulate it into our setting.
\begin{definition}\label{tauto_def} Consider a CRN $(\cG,\k)$ with stochastic dynamics on $\Ga$ and $\pi$ a distribution on $\Ga$.
We say $(\cG,\k)$ is generalized balanced for $\pi$ on $\Ga$ if there exists $\{(L_i,R_i)_{i\in A}\}$ a set of tuples of subsets of $\cR$ with
$$\dot\bigcup_{i\in A}L_i=\dot\bigcup_{i\in A}R_i=\cR $$ such that for all $i\in A$ and all $x\in \Ga$ we have
\begin{equation}\label{gen_cxbal}
  \sum_{\nu\to\nu\p\in L_i} \pi(x+\nu-\nu\p)\la_{\nu \to \nu\p}(x+\nu -\nu\p)=\pi(x)\sum_{\nu\to\nu\p R_i}\la_{\nu \to \nu\p}(x).
\end{equation}
\end{definition}
\begin{remark} The notion of generalized balanced covers 
\begin{enumerate}
\item reaction balanced with index given by reactions, i.e. the tuples of subsets are $\{(\nu\to\nu\p,\nu\p\to\nu)_{\nu\to\nu\p\in\cR}\}$
\item complex balanced with index given by complexes, i.e. the tuples of subsets are defined for $C\in\cC$\hfill\break $L_C=\{\nu\to\nu\p\in\cR|\nu=C\},R_C=\{\nu\to\nu\p\in\cR|\nu\p=C\}.$
\item reaction vector balanced with index given by $a\in \Z^n$, i.e. the tuples of subsets are defined for $a\in \Z^n$\hfill\break $L_a=\{\nu\to\nu\p\in\cR|\nu-\nu\p=a\},R_a=\{\nu\to\nu\p\in\cR|\nu-\nu\p=-a\}.$
\end{enumerate}
of \cite{JoshiCap}, including combinations and other possibilities (see e.g. Remark \ref{re_cx_bal}).
\end{remark}
The following Proposition generalizes \cite[Theorem 4.3]{JoshiCap} and follows from the same principle applied to the system of equations defining the Master equation.
\begin{proposition}\label{tauto_lemma}
If $(\cG,\k)$ is a CRN with stochastic dynamics on $\Ga$ that is generalized balanced for $\pi$, then $\pi$ is a stationary distribution for $(\cG,\k)$ on $\Ga$.
\end{proposition}
\begin{proof} 
We have to check that the master equation is satisfied for all $x\in \Ga$, so consider a fixed $x\in \Ga$.
By definition we have a decomposition of the reactions of the form $\{(L_i,R_i)_{i\in A}\}$ with $\dot\bigcup_{i\in A}L_i=\dot\bigcup_{i\in A}R_i=\cR $. For a specific $i\in A$ we then have
$$
  \sum_{\nu\to\nu\p\in L_i} \pi(x+\nu-\nu\p)\la_{\nu \to \nu\p}(x+\nu -\nu\p)=\pi(x)\sum_{\nu\to\nu\p R_i}\la_{\nu \to \nu\p}(x).
$$
Since the original master equation \ref{master_eq1} is comprised of these equations we conclude that $\pi$ is a stationary measure for $(\cG,\k)$ on $\Ga$.
\end{proof}

\section{Autocatalytic CRNs}\label{model}
The class of autocatalytic reaction networks we study is a relatively broad class of mass-preserving non-weakly reversible CRNs of arbitrary deficiency. It is inspired by both the inclusion process \cite{inclusion_proc} and the Misanthrope process \cite{Cocozza-Thivent1985} with nontrivial intersection (also see $\S$ \ref{s.inclusion}). Therefore it naturally generalizes both models studied in CRN literature \cite{Swarm,Bianc,Bianc2,Houch,Saito1,Saito2} and some models of homogeneous and inhomogeneous interacting particle systems on finite lattices \cite{inclusion_proc,liggett1997}.  
\subsection{Notations}
All reactions in autocatalytic CRNs will have a net consumption of one $S_i$ and a net production of one $S_j$ and will be of the following form
\begin{equation} \label{basic_re}S_i+(m-1)S_j\to mS_j, \end{equation}
where $m\geq1$. 
 We use the following notation for the reaction rates for such reactions
 $$\a^1_{i,j}=\text{rate of the reaction } S_i\to S_j$$
  $$\a^m_{i,j}=\text{rate of the reaction } S_i+(m-1)S_j\to mS_j$$
Summarizing this information with a vector we write \begin{equation}\label{vect_not} \a_{i,j}:=(\a^1_{i,j},\cdots,\a^{n_{i,j}}_{i,j})\end{equation} 
where $n_{i,j}$ is the highest integer $m$ with 
a reaction of the form $S_i+ (m-1)S_j\to mS_j$. \\
Denote the collection of reactions net consuming one $S_i$ and net producing one $S_j$ by \begin{equation}\label{part}\cR_{i,j}=\cR_{-e_i+e_j}:=\{\nu\to\nu\p\in\cR:\nu\p -\nu= e_j-e_i\}.\end{equation}

\subsection{Autocatalytic reaction network }
\begin{definition}\label{autocatalyticCRN}
A  CRN is said to be  autocatalytic 
(see Remark \ref{abbr}) denoted by $(\cG^*,\k)$ in what follows,  when $\cG^*=(\cS,\cC,\cR)$ on the species $\cS=\{S_1,\cdots ,S_n\}$ satisfies the following rules:
  \begin{enumerate}
\item All reactions are of the form \eqref{basic_re}.
\item If there is a reaction net consuming one $S_k$ and net producing one $S_l$, then \hfill\break$S_k\to S_l,S_l\to S_k\in \cR$. (mass-exchange in both directions, no absorption)
\item There is one monomolecular linkage class.
\item If $S_j\to S_k \in \cR_{j,k}\subset\cR$ and $S_l\to S_k\in\cR_{l,k}\subset \cR$ then the reactions in $\cR_{j,k}$ and $\cR_{l,k}$
contain reactions of the same molecularity
such that there is a $c\in\R_{>0}$ with
$$c\cdot (\a^1_{j,k},\cdots,\a^{n_{j,k}}_{j,k})=(\a^1_{l,k},\cdots,\a^{n_{l,k}}_{l,k}).$$
Set for convenience $n_k:=n_{j,k}=n_{l,k}$,
and denote the normalised rates by
$$(1,\be^2_{k},\cdots,\be^{n_k}_{k}):=\frac{1}{\a^1_{j,k}}(\a^1_{j,k},\a^2_{j,k},\cdots,\a^{n_k}_{j,k}) $$
where $n_k$ is the highest integer with 
a reaction of the form $S_j+ (n_k-1)S_k\to n_kS_k$.
\item There is a vector $\la\in \R_{>0}^n$, such that 
$\la$ is an stationary distribution for the reversible Markov chain of transition kernel $Q=(\a_{i,j}^1)_{i\ne j\in S}$, that is, 
$$\la_i \a_{i,j}^1=\la_j  \a_{j,i}^1,$$
$\forall i\ne j$ with $S_i\to S_j, S_j\to S_i\in \cR$.

\end{enumerate}
\end{definition}
\begin{remark}\label{irr_comp} All autocatalytic CRNs are mass-preserving, 
 meaning that every reaction $\nu\to \nu\p$ of the CRN satisfies $\sum_{i\in\cS}\nu_i =\sum_{i\in\cS}\nu\p_i .$
Hence the stochastic dynamics are confined to irreducible components of the form $\Ga_N:=\{x\in\Z_{\geq 0}^n| |x|=N\}$; similarly the deterministic ODE dynamics are restricted to corresponding stochiometric compatibility classes. Furthermore (1) of definition \ref{autocatalyticCRN} means that only monomers are exchanged in a particle system interpretation, while (4) \& (5) ascertain that the CRN has a product-form stationary distribution (see proof of Theorem \ref{prod_th}).
\end{remark}

\begin{remark}\label{abbr}
Note that this expression was already used in different contexts. A definition of autocatalytic CRNs can be found for weakly-reversible CRNs in \cite{Manoj1} where it is utilized in the study of persistence and siphons for such CRNs. Other definitions of autocatalytic reaction and autocatalytic set can be found in numerous references, most of them focusing on the framework of origin of life (see e.g. \cite{autocat_2,autocat_1}), whose examination in this context can be traced back to Kauffmann \cite{KAUFFMAN19861}.  
\end{remark}

\subsection{Examples of autocatalytic reaction networks \label{s.swarm}}\label{Table1}
Here we introduce the notions and illustrate applications and the model. For the CRNs here on two species, conditions (1)-(4)  of Definition \ref{autocatalyticCRN} are easily seen to be satisfied, and condition (5) is trivial.  For a frameworks of interest for autocatalytic CRNs with more species we refer to $\S$ \ref{s.inclusion}.
\begin{table}[H]
\centering
\begin{tabular}{ |c|c|c|c| }
 \hline Example (A)&Example (B)&Example(C)&Example(D)
 \\ 
 \hline
 \begin{tikzcd} [ row sep=1em,
  column sep=1em]
S_1\arrow[r, shift left=1ex,"\a_{1,2}^1"]  &S_2 \arrow[l,"\a_{2,1}^1"]\\ 
\end{tikzcd}
 &
\begin{tikzcd} [ row sep=1em, column sep=1em]
S_1\arrow[r, shift left=1ex,"\a_{1,2}^1"]  &S_2 \arrow[l,"\a_{2,1}^1"]\\ 
 mS_1+S_2 \arrow[r, "\a_{2,1}^{m+1}"]  &(m+1)S_1
\end{tikzcd}
 & 
$\begin{tikzcd} [ row sep=0.6em,
  column sep=0.6em]
S_1\arrow[r, shift left=1ex,"\a_{1,2}^1"]  &S_2 \arrow[l,"\a_{2,1}^1"]\\ 
  &2S_1 \\ 
S_1+S_2 \arrow[ru, "\a_{2,1}^2"] \arrow[dr,"\a_{1,2}^2"] \\
& 2S_2
\end{tikzcd}
$ 
&
\begin{tikzcd} [ row sep=0.7em, column sep=0.7em]
S_1\arrow[r, shift left=1ex,"\a_{1,2}^1"]  &S_2 \arrow[l,"\a_{2,1}^1"]\\ 
2 S_1+S_2 \arrow[r, "\a_{2,1}^3"]  &3S_1 \\ 
  S_1+2S_2\arrow[r,"\a_{1,2}^3"] &3S_2\\
    &2S_1 \\ 
S_1+S_2 \arrow[ru, "\a_{2,1}^2"] \arrow[dr,"\a_{1,2}^2"] \\
& 2S_2
\end{tikzcd}
\\ 
 \hline
\end{tabular}
\caption {Some autocatalytic CRNs (Definition \ref{autocatalyticCRN})  drawn via reaction graph} \label{table2} 
\end{table}

All examples are autocatalytic CRNs (Definition \ref{autocatalyticCRN}).
\textbf{Example (A)} is reversible and of deficiency 0 and coincides with motif E of \cite{Saito2}. \textbf{Example (B)} contains asymmetric transitions, is non-weakly reversible with deficiency 1 and corresponds to motif  F of \cite{Saito2}.
\textbf{Example (C)} is non-weakly reversible with deficiency 2 and is a generalized model of \cite{Bianc},\cite{Swarm},\cite{Saito1}, 
which also appears as a special case of motif I of \cite{Saito2}.

We next remark on some applications of autocatalytic reaction networks.
Examples (C) and (D) have found applications in several interdisciplinary fields. Example (C) can model a colony of foraging ants collecting food from two sources \cite{Bianc2}, it was exploited for decision-making processes in a swarm of agents \cite{Swarm} and apart from that corresponds to the Moran model on two competing alleles with bidirectional mutation \cite{popo,Houch}. Example (D) was introduced  as a high-density model for decision-making processes in swarms of agents and ants \cite{Swarm}. Then the trimolecular reactions of Example (D) model the majority rule, where the majority convinces the minority to change its opinion in collective decision making systems (or food source in ants). We provide the stationary distribution in closed form in Theorem \ref{prod_th} for all autocatalytic CRNs, leading to exact known stationary behaviour in all examples above. 
\section{Product-form stationary distributions for autocatalytic CRNs}\label{exa}

\subsection{A non-standard product-form stationary distribution}

Here we derive product-form stationary distributions for autocatalytic CRNs (see Definition \ref{autocatalyticCRN}). This class of CRNs and Theorem \ref{prod_th} is stimulated both by the inclusion process \cite{inclusion_proc} and the Misanthrope process \cite{Cocozza-Thivent1985} and contains models studied in the CRN literature \cite{Swarm,Bianc,Bianc2,Houch,Saito1,Saito2} and models of homogeneous and inhomogeneous interacting particle systems on finite lattices \cite{inclusion_proc,liggett1997}. For a proof in the Misanthrope case see e.g. \cite{chleb}.

\begin{theorem}\label{prod_th} Let $(\cG^*,\k)$ be an autocatalytic CRN  (see Definition \ref{autocatalyticCRN}).
Then the associated stochastic CRN has its stochastic dynamics confined to irreducible components of the form $\Ga_N:=\{x\in\Z_{\geq 0}^n| |x|=N\}$, 
is reversible and possesses the product-form stationary distribution
\begin{equation}\label{MainFormula}
\pi(x)=Z_\Ga^{-1}\prod_{S_i\in\cS}^{ }f_i(x_i) ,
\end{equation}
with product-form functions
$$f_i(x_i)=\lambda_i^{x_i}p_i(x_i)$$ where
\begin{equation}\label{fi}
p_i(m)=\frac{1}{m!}\prod_{l=1}^m(1+\sum_{k=2}^{n_i}\be_i^k \prod_{r=1}^{k-1}(l-r)).
\end{equation}

\end{theorem}
\begin{proof} 
First remark that condition (5) of Definition \ref{autocatalyticCRN} holds
if and only if for each $i,j$ such that $\cR_{i,j}\neq 0$ (of \eqref{part}) there exists  a $c(i,j)>0$ such that 
\begin{equation}\label{cond_pairw} \a_{i,j}^1=\frac{\la_j}{c(i,j)},\a_{j,i}^1=\frac{\la_i}{c(i,j)}.\end{equation}
We show that $\pi$ is reaction vector balanced for any irreducible component $\Ga$ by separating the master equation into parts according to reaction vector balance (\ref{ReacBal}). 
According to conditions (1) and (2) given in Definition \ref{autocatalyticCRN}, we can partition the set of reactions using the various sets 
$\cR_{i,j}$ and $\cR_{j,i}$ (of \eqref{part}), and hence subdivide the master equation according to this partitioning. 
Let $i,j$ be such that $S_i\to S_j\in\cR_{i,j}\subset \cR $. 
\begin{claim}  $\pi$ as defined in (\ref{MainFormula}) satisfies the respective equation (\ref{ReacBal}) associated to $\cR_{i,j}$, for all $x\in \Ga\subset\Z_{\geq 0}^n$.
\end{claim}
\begin{proof}
In the following we omit the coefficients $x_l$ for $l\neq i,j$ in the equation from $\pi$, since the other coordinates are equal and we prove $\pi$ has product-form. We only get reactions $\cR_{i,j}$ on the left side and reactions $\cR_{j,i}$ on the right side of (\ref{ReacBal}): we must thus check that the $f_i$ solve
\begin{equation}\label{react_}\pi(x_i+1,x_j-1)(x_i+1)\big (\sum_{l=1}^{n_j}\a_{i,j}^l\1_{\{x_j\geq l \} }\prod_{k=1}^{l-1}(x_j-k)\big ) \end{equation}
$$=\pi(x_i,x_j)x_j\big ( \sum_{q=1}^{n_i}\a_{j,i}^q\1_{\{x_j\geq 1,x_i\geq q-1 \} }\prod_{m=1}^{q-1}(x_i+1-m)\big )$$
Observe that this equation vanishes on both sides for $(x_i,x_j)=(x_i,0)\in \Z_{\geq 0}\times \{0\}$, and that for all $(x_i,x_j)\in \Z_{\geq 0}\times \Z_{\geq 1}$ we have
$$\1_{\{x_j\geq l \}}\prod_{k=1}^{l-1}(x_j-k)=\prod_{k=1}^{l-1}(x_j-k)$$
$$\1_{\{x_j\geq 1,x_i\geq q-1 \} }\prod_{m=1}^{q-1}(x_i+1-m)=\prod_{m=1}^{q-1}(x_i+1-m)$$
where one can reduce the second identity to the first on the domain we consider. Set (both for $i,j$)
$$g_i(m)=\frac{1}{m!}\prod_{l=1}^m(\sum_{k=1}^{n_i}\a_{j,i}^k\prod_{r=1}^{k-1}(l-r)).$$
Then for $(x_i,x_j)\in \Z_{\geq 0}\times \Z_{\geq 1}$ we get
$$\big (\sum_{l=1}^{n_j}\a_{i,j}^l\1_{\{x_j\geq l \} }\prod_{k=1}^{l-1}(x_j-k)\big )=\frac{x_j\cdot g_j(x_j)}{ g_j(x_j-1)},$$
$$\big ( \sum_{q=1}^{n_i}\a_{j,i}^q\1_{\{x_j\geq 1,x_i\geq q-1 \} }\prod_{m=1}^{q-1}(x_i+1-m)\big )=\frac{(x_i+1)\cdot g_i(x_i+1)}{ g_i(x_i)}.$$

Next inserting $\pi(x_i,x_j)=g_i(x_i)g_j(x_j)$ in (\ref{react_}) we obtain  
$$g_i(x_i+1)g_j(x_j-1)(x_i+1)\frac{x_j\cdot g_j(x_j)}{ g_j(x_j-1)}=g_i(x_i)g_j(x_j)x_j\frac{(x_i+1)\cdot g_i(x_i+1)}{ g_i(x_i)}$$
By shortening fractions this is equivalent to 
$$g_i(x_i+1)\cdot g_j(x_j)\cdot (x_i+1)\cdot x_j=x_j\cdot g_j(x_j)\cdot(x_i+1)\cdot g_i(x_i+1)$$
so this Ansatz solves the equation.
Observe that along equations (\ref{react_}) $x_i+x_j$ is the same on the left and on the right hand side, so any functions 
$$h_i(m)=d^m\cdot g_i(m),h_j(m)=d^m\cdot g_j(m) ,d>0$$
are also solutions to (\ref{react_}). 
However we have to choose the product-form functions compatible taking into account all $i,j$ with $\cR_{i,j}\neq \emptyset$. Hereby we shall show that for all $i,j$ with $\cR_{i,j}\neq \emptyset$ we find a $d(i,j)>0$ such that we arrive at the same product-form functions and that they correspond to $f_i$. \hfill\linebreak 
 For this we use (\ref{cond_pairw}) to set $$d(i,j)=c(i,j)=\frac{\la_j}{\a^1_{i,j}}=\frac{\la_i}{\a^1_{j,i}}.$$
 Then the $g_i(m)$ can be written as
 $$g_i(m)=(\a_{j,i}^1)^m p_i(m)$$
 where the $p_i$ are defined as
$$p_i(m)=\frac{1}{m!}\prod_{l=1}^m(1+\sum_{k=2}^{n_i}\be_i^k \prod_{r=1}^{k-1}(l-r)).$$
With this we write
\begin{equation}\label{change_} g_i(m)\cdot c(i,j)^m=(\frac{\la_i}{\a_{j,i}^1})^m(\a_{j,i}^1)^m p_i(m)=\la_i^m p_i(m):=f_i(m).\end{equation}
as required.  Notice that the $f_i(m)$ as the resulting product-form functions are well-defined and do not depend on specific pairs $i,j$, using both condition (4) and (5) from definition \ref{autocatalyticCRN}.
\end{proof}
\end{proof}

\begin{remark}\label{re_cx_bal}
Notice that autocatalytic CRNs considered in Theorem \ref{prod_th} can be combined with complex balanced CRNs to obtain a bigger class of CRNs for which the stationary distribution is known and of product-form. This is thanks to the product-form and Proposition \ref{tauto_lemma}. The incoming reactions in the autocatalytic part which are also part of a complex balanced CRNs are however restricted to be monomolecular.
\end{remark}
We give an example to outline this and indicate the principle.
\begin{example}\label{ex_cx_bal} In this example the CRN is composed of the upper part which is reaction vector balanced and corresponds to reactions between $S_1,S_2$ and the lower part which is complex balanced and corresponds to reactions between $S_1,S_3$.
$$\begin{tikzcd} [ row sep=1em,
  column sep=1em]
S_1\arrow[r, shift left=1ex,"\a_{1,2}^1"]  &S_2 \arrow[l,"\a_{2,1}^1"]\\ 
 S_1+S_2 \arrow[r, "\a_{1,2}^2"]  &2S_2 \\
 2S_1 \arrow[rr, shift left=1ex,"\k_1"] 
&& 2S_3 \arrow[dl, "\k_2"] \\
& S_1+S_3\arrow[ul,"\k_3"]
\end{tikzcd}
$$

The stationary distribution is
$$\pi(x_1,x_2,x_3)=\frac{(\a_{2,1}^1)^{x_1}}{x_1!}\frac{\prod_{j=1}^{x_2}(\a^1_{1,2}+(j-1)\a^2_{1,2})}{x_2!}\frac{(c_3\p)^{x_3}}{x_3!}$$
on irreducible components of the form $$\Ga_N=\{x\in\Z_{\geq 0}^3| \sum_{i=1}^3x_i=N \}$$
with $c_3\p=c_3\cdot \frac{\a_{2,1}^1}{c_1}$, where $(c_1,c_3)$ is a point of complex balance of the lower CRN (i.e. complex balanced for the CRN that consist only of reactions between $S_1,S_3$). 
Since the balance equation for the upper CRN are reaction vector balanced, while the lower are complex balanced, the CRN is overall generalized balanced on $\Ga_N,N\geq 2$(see Definition \ref{tauto_def}).

\end{example}

\subsection{Asymptotic behaviour of product-form functions of Theorem \ref{MainFormula}\label{quant_ana}}
The product-form functions which appear in Theorem \ref{MainFormula} are of the form
$$f_i(m)=\frac{\la_i^m}{m!}\prod_{l=1}^m(1+\sum_{k=2}^{n_i}\be_i^k \prod_{r=1}^{k-1}(l-r)).$$
Here we use of the following notations ( for 1-/2-/3-molecular incoming reaction):
\begin{enumerate}
\item $g(m)=\frac{(\la_i)^m}{m!}\prod_{l=1}^m(1+0)=\frac{(\la_i)^m}{m!}$
\item $h(m)=\frac{(\la_i)^m}{m!}\prod_{l=1}^m(1+\be_i^2(l-1))$
\item $p(m)=\frac{(\la_i)^m}{m!}\prod_{l=1}^m(1+\be_i^2(l-1)+\be_i^3(l-1)(l-2))$
\end{enumerate}

We study asymptotic growth behavior and the problem of normalizability of the different product-form functions. These considerations will be related to condensation in $\S$ \ref{conde_}. Identifying the product-form functions with sequences, the latter is equivalent to existence of finite positive radius of convergence of the associated power series. We say a sequence $(a_n)_n \in \R_{\geq 0}^\N$ is normalizable if there is $c>0$ such that $\sum_{n=0}^\infty a_nc^n<\infty$. We omit the proof of the following lemma which is standard. 

\begin{lemma} (asymptotic growth behavior)\label{growth_lemma}
\begin{enumerate}
\item $\lim\limits_{n \to \infty}\frac{g(n+1)}{g(n)}=0$.
\item $\lim\limits_{n \to \infty}\frac{h(n+1)}{h(n)}=\la_i\be_i^2$.
\item $\frac{p(n+1)}{p(n)}\to \infty$.
\item The same quotient of product-form functions coming from molecularity higher than 3 also diverges.
\end{enumerate}
In particular the limits of (1), (3) and (4) do not depend on the parameter $\la_i$.
\end{lemma}
Via ratio test we get that only the power series of the functions $h(n)$ have finite positive radius of convergence $(\la_i\be_i^2)^{-1}$ of the associated power series.  $g(n)$ has infinite convergence radius and $p(n)$ has a convergence radius of zero. 
\begin{lemma} (Normalizability of product-form functions)\label{norm_lemma}
\begin{enumerate}
\item $\sum_{m=0}^{\infty}\phi^mg(m)=e^{\phi\la_i}<\infty$ for all $\phi\in \R$.
\item $\sum_{m=0}^{\infty}\phi^mh(m)=\sum_{m=0}^{\infty}\phi^m\frac{(\la_i\be_i^2)^m}{m!}\frac{\Ga(\frac{1}{\be_i^2}+m)}{\Ga(\frac{1}{\be_i^2})}=(1-\la_i\be_i^2\phi)^{-\frac{1}{\be_i^2}}<\infty$ \\
for $0<\phi<(\la_i\be_i^2)^{-1}$
\item $(\sum_{m=0}^{n}\phi^m p(m))_n$ does not converge for any $0<\phi$
\item product-form functions coming from molecularity higher than 3 are also not normalizable as in 3.
\end{enumerate}
\end{lemma}
We have three different behaviors with respect to normalizability, $g(m)\phi^m$ can be normalized for any $0<\phi$, $h(m)\phi^m$ can only be normalized up to $0<\phi<0(\la_i\be_i^2)^{-1}$ whereas $p(m)\phi^m$ can not be normalized independent of the value $0<\phi$. 
Due to the conservative nature of mass-preserving CRNs or conservative interacting particle systems (IPS), rescaling all the product-form functions by the same $\phi$ does not change the distribution (for stochastic particle systems this parameter $\phi$ is called fugacity \cite{inclusion_proc,monotonic_1}).

\subsection{The classical mean field scaling\label{s.classicalscaling}}\hfill\break
Denote by $|\nu|=\sum_{S_i\in \cS}\nu_i$ the number of molecules involved in a reaction, and designate by $V$ the scaling parameter usually taken to be the volume times Avogadro's number. Then in some situations it is reasonable to rescale the transition rates of the stochastic model according to the volume as
\begin{equation} \label{rate_resc}
\la_{\nu \to \nu\p}^V(x)=\frac{V\k_{\nu\to\nu\p}}{V^{|\nu|}}\frac{(x)!}{(x-\nu)!}\1_{x\geq \nu},
\end{equation}
corresponding to the following change of the reaction rate 
$$\k_{\nu\to\nu\p}\to \tilde{\k}_{\nu\to\nu\p}=\frac{V\k_{\nu\to\nu\p}}{V^{|\nu|}}.$$
This way of rescaling the transition rates is adopted by considering the probability that a set of
$\vert\nu\vert$ molecules meet in a small volume element to react \cite{anderson2,kurtz2}. The above mean field scaling assumes that a particular  molecule of type $S_i$ will meet a molecule of type $S_j$ with a probability proportional to the concentration of type $S_j$ molecules. 
 Kurtz \cite{kurtz2} linked the short term behavior of the properly scaled continuous-time Markov chain to the dynamics of the ODE model. 
Within the classical scaling regime, Theorem \ref{prod_th} becomes
\begin{theorem}\label{prod_th_ad} Let $(\cG^*,\k)$ be  autocatalytic  (see Definition \ref{autocatalyticCRN}).
Then the associated stochastic CRN, with rate function as in \ref{rate_resc} possess the product-form stationary distribution
\begin{equation}\label{MainFormula_2}
\pi(x)=Z_\Ga^{-1}\prod_{i\in\cS}^{ }f_i(x_i),
\end{equation}
with product-form functions
$$f_i(m)=\lambda_i^{m}\frac{1}{m!}\prod_{l=1}^m(1+\sum_{k=2}^{n_i}\frac{\be_i^k}{V^{k-1}} \prod_{r=1}^{k-1}(l-r))$$ 
with the stochastic dynamics confined to irreducible components $\Ga$ as specified in Remark \ref{irr_comp}
\end{theorem}
It is then natural to check the large $V$ behaviour of the stationary distribution given in Theorem \ref{prod_th_ad}. 
Recently, large deviation theory has been developped for some class of strongly endotactic mean field CRNs in \cite{eckmann}, but these results do not 
apply to autocatalytic networks.
We will consider a non mean field regime in $\S$ \ref{sect_cond} and
 just illustrate the mean field scaling limit  by an application of \cite[Theorem 3.1]{chan} to the following example. 
Consider the autocatalytic CRN, 
$$
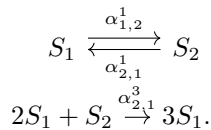
\begin{tikzcd} [ row sep=1em, column sep=1em]
S_1\arrow[r, shift left=1ex,"\a_{1,2}^1"]  &S_2 \arrow[l,"\a_{2,1}^1"]\\ 
  2S_1+S_2\arrow[r,"\a_{2,1}^3"] &3S_1.
\end{tikzcd}
$$ 
Let $V=N$ be the total number of molecules, and 
let $X(t)$  be the number of molecules of type $S_2$ at time $t\ge 0$, which is a birth and death process evolving in the set $\{0,\cdots,N\}$. It is easy to see that the conditions for application of \cite[Theorem 3.1]{chan} are satisfied. 

Then the rescaled process $Y_N(t)=X(t)/N$  approaches the dynamics of the mass action ODE $dy/dt =F(y)$, where $F(y)=b(y)-d(y)$, with $b(y)= \a_{1,2}^1 (1-y)$ and $d(y)=\a_{2,1}^1 y +\a_{2,1}^3 y (1-y)^2$, $y\in [0,1]$. 
Focusing on $Y_N$, the stationary distribution given in Theorem \ref{prod_th_ad}
translates into a probability measure $\pi_N$ defined on the unit interval $[0,1]$, which 
satisfies a large deviation priciple for this invariant probability measure \cite[Theorem 3.1]{chan},
%
where the stationary distribution concentrates exponentially fast as $N\to\infty$ on the set of minimizers of the free energy function, which are precisely the linearly stable equilibria of the associated deterministic mass action dynamic. 

One can check that for generic constants $\a_{1,2}^1$, $\a_{2,1}^1$ and $\a_{2,1}^3$ the mass action ODE has a single stable equilibrium which is located in the positive orthant; this follows since it is enough to confirm it for $dy/dt=F(y)$ as above (polynomial in one variable). $\S$ \ref{sect_cond}
considers a different scaling regime  for autocatalytic processes in which condensation occurs. 
In the above example, the stationary distribution converges to the point mass $\delta_0$ centered at $y=0$, see Theorem \ref{SecondCondensationResult} and Corollary \ref{Corollarycondensation}.


\section{Application: Condensation in particle systems\label{sect_cond}}
We investigate the asymptotic behaviour of mass preserving autocatlytic networks when the total number of molecules $N$ is large. When considering large volume limits, CRN theory usually considers the 
classical mean field scaling limit, see $\S$ \ref{s.classicalscaling}. We focus on a different mechanism, that leads to a CRN (or a particle system) where molecules do not move at random in a mean field regime, but are located at the nodes of a graph. Molecules located at some node $i$ (or  of type $S_i$)  can move to nearest neighbour sites $j$. In this modeling framework the rate at which a molecule of type $S_i$ moves to site $j$ (or is converted into a molecule of type $S_j$) will be function of the absolute number of particles of type $S_i$ and $S_j$, so that the rate constant $\kappa_{\nu\to\nu\p}$ will be independent of $N$. 

This will model the autocatalytic effect where the move of a molecule from site $i$ to site $j$ is a consequence of the attraction of molecules of type $j$ on molecules of type $i$. In this setting, a new phenomenon appears: under some conditions, the molecules will concentrate on a subset of the set of species, leading thus to condensation on a subset of the state space. We first illustrate this phenomenon by considering the so-called {\it inclusion process}. We then study condensation by investigating the asymptotic behavior of the product-form stationary distribution $\pi_N$, putting emphasis on the cases of up to molecularity three. We introduce three different forms of condensation and investigate the limiting distributions for autocatalytic CRNs. We observe that monomolecular autocatalytic CRNs (see Definition \ref{autocatalyticCRN} and Remark \ref{abbr}) and complex balanced CRNs do not satisfy any form of condensation. We prove for the up to bimolecular case a weak form of condensation and a weak law of large numbers. In the threemolecular and higher case we show that such systems exhibit the strongest form of condensation. 
\subsection{Condensation in inclusion processes\label{s.inclusion}}

The inclusion process, introduced in \cite{Giardina2009,Giardina2010}, is a particle system which is dual to the {\it Brownian Energy Model} where every particle of type $S_i$ can attract particles from type $S_j$ at rate $p_{ji}$, where the $p_{ij}$ are the transition probabilities of a Markov chain. When $p_{ij}=p_{ji}$, one speaks of symmetric inclusion process (SIP). This particle system evolves in 
$\Z_{\geq 0}^{\cS}$, where $\cS$ is the set of species. It is defined as a time-continuous Markov chain of generator $\mathcal{L}$ of the form
$$\mathcal{L}h(x)=\sum_{i\ne j} p_{ij}x_i (\frac{m}{2}+x_j)(h(x+e_j-e_i)-h(x)),$$
where $h$ denotes any function. In the homogeneous case this is a special case of the Misanthrope process on a finite lattice \cite{Cocozza-Thivent1985}. The symmetric inclusion process defines in fact a stochastic reaction network for the set of reactions $R_{ij}$ given by
$$
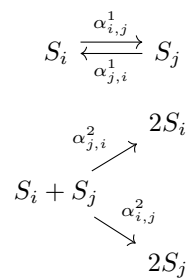
\begin{tikzcd} [ row sep=1em,
  column sep=1em]
S_i\arrow[r, shift left=1ex,"\a_{i,j}^1"]  &S_j \arrow[l,"\a_{j,i}^1"]\\ 
  &2S_i \\ 
S_i+S_j \arrow[ru, "\a_{j,i}^2"] \arrow[dr,"\a_{i,j}^2"] \\
& 2S_j
\end{tikzcd}
$$
with $\a_{i,j}^1 = p_{ij}\frac{m}{2}$ and $\a_{i,j}^2=p_{ij}$.
$R_{ij}$ is thus a multi-species version of example (C) of $\S$
\ref{s.swarm}.
The authors of \cite{inclusion_proc} studied such processes and provided interesting results on asymmetric CRNs. Notice that such CRNs can be autocatalytic when the Markov chain of transition probabilites $p_{ij}$ is reversible and when conditions (4) and (5) of Definition \ref{autocatalyticCRN} are satisfied. Such process are mass-conservative. Let $N$ be the total number of particles, and let $\pi^N$ be the stationary distribution associated with the process restricted to the irreducible component $\Lambda_N =\{x\in\Z_{\geq 0}^{\cS};\ \sum_{i\in \cS} x_i =N\}$.

The authors of \cite{inclusion_proc} provide an interesting one dimensional process, called asymmetric inclusion process (ASIP),
where $p_{i i+1}=p$ and $p_{i i-1}=q$ on the state space
$\cS=\{1,\cdots,n\}$ with factorised stationary distributions as in Theorem \ref{prod_th} with $\la_i = (p/q)^i$ of (5) of Definition \ref{autocatalyticCRN}.

A new interesing phenomenon appears in such process: In the limit $N\to\infty$ and when $p>q$, the process condensates on the right edge, that is $\pi^N(X_n \le (1-\delta)N)\longrightarrow 0$, for all $\delta \in (0,1)$. The authors argued that at first sight one  might be tempted to think that this is just a consequence of the asymetry $p>q$, and proved that this argument is not correct since a CRN having the same coefficient $\a_{i,j}^1$ but vanishing second-order coefficient $\a_{i,j}^2\equiv 0$ would have a Poissonian product-form stationary distribution and no condensation would occur.

Building on this work, the  authors of \cite{Bianchi16} considered a reversible inclusion process which is reversible as a Markov chain with
 $\la_i p_{ij}\equiv \la_j p_{ji}$ (as in (5) of Definition \ref{autocatalyticCRN}),
where the diffusion constant $m_N$ depends on the total number of particles $N$ in such a way that
$m_N \ln(N)\longrightarrow 0$ as $N\to\infty$. They proved that the process condensates on the set of species 
where the stationary distribution $\la$ attains its maximum value.
 We will extend these results to autocatalytic CRNs of arbitrary molarity.

\subsection{Condensation in autocatalytic reaction networks}\label{conde_}
Consider a sequence of random vectors $(X_N)_{N\in\N}$ indexed by $N$, where $X_N=(X_1,\cdots ,X_n)_N$ takes values in
\begin{equation}\label{abs_irr}\Ga_N:=\{x\in\Z^n_{\geq 0}\text{ such that } |x|=N\}.\end{equation}
Denote the corresponding sequence of discrete probability distributions by $\pi_N$, i.e.
\begin{equation}\label{abs_cond}\pi_N(x):=P(X_N=x),x\in\Z^n_{\geq 0}\end{equation}
We use this 
setting to first make some general observations and statements. Let $[n]:=\{1,\cdots,n\}$ and denote the coordinate-wise maximum and projection of an element $x\in\Z^n_{\geq 0}$ by
$$\M(x):=\max\limits_{i\in [n]}x_i \text{ and } \text{proj}_i(x):=x_i\text{ for }i\in [n]
$$
We allow the following abuse of notation for simplicity, where $q:\R\to\R$ is a function, and write
$$ \pi_{N}(X_j\geq q(N)):=P(\text{proj}_j(X_N)\geq q(N))=\pi_N(\{x\in\Z^n_{\geq 0}|x_j\geq q(N)\})$$

Following \cite{inclusion_proc,monotonic_1,Bianchi16} we introduce three notions for condensation.

\begin{definition}\label{condens_ab} In the setting of (\ref{abs_cond}) we define the following notions of condensation
 \begin{enumerate}[label=(C\arabic{*})]
  \item \label{C1} $\lim\limits_{N \to \infty} \pi_{N}(\M(X_N) = N)=1$
  \item \label{C2} $\lim\limits_{K \to \infty}\lim\limits_{N \to \infty} \pi_{N}(\M(X_N)\geq N-K)=1$
  \item \label{C3} For all $\delta \in (0,1)$ we have 
$\lim\limits_{N \to \infty} \pi_{N}(\M(X_N)\geq \delta N)=1$
 \end{enumerate}
\end{definition}
Next, we show (C1) is the strongest and (C3) is the weakest notion given in definition \ref{condens_ab}, the simple proof is omitted.
\begin{lemma}\label{cond_lem_1} We have the implications:
$$(C1)\implies(C2)\implies(C3)$$
\end{lemma}
Notice the following sufficient conditions for forms of condensation of definition \ref{condens_ab}. 
\begin{lemma}\label{cond_rem_1} 
Assume there are $k>0$ different coordinates of $X_N$ which are denoted by the set $B\subseteq [n]$ (with $|B|=k$),  such that one of the following holds 
\begin{enumerate}
  \item[(1)] For all $j\in B$ $\lim\limits_{N \to \infty} \pi_{N}(X_j = N\})=\frac{1}{k}$
  \item[(2)] For all $j\in B$ $\lim\limits_{K \to \infty}\lim\limits_{N \to \infty} \pi_{N}(X_j\geq N-K\})=\frac{1}{k}$
  \item[(3)] For all $j\in B$ for all $\delta$ where $1 > \delta \geq \frac{(n-1)}{n}$ we have \\
$\lim\limits_{N \to \infty} \pi_{N}(X_j\geq \delta N\})=\frac{1}{k}$
 \end{enumerate}
 Then if (1) holds this implies (C1), if (2) holds this implies (C2) and if (3) holds this implies (C3). 
  \end{lemma}
 \begin{remark}\label{cond_conditioning}
 If a random vector $X=(X_1,\cdots, X_n)$ takes value in $\Z^n_{\geq 0}$, then conditioning on the sum being $N$ gives a sequence of random variables as in \eqref{abs_cond}.
 \end{remark}



 Both inclusion process on $\Z_{\geq 0}^\cS$(or other conservative IPS) and mass-preserving CRNs (see Remark \ref{irr_comp}) are continuous-time Markov chains with positive recurrent stationary stochastic dynamics confined to finite sets of the form (\ref{abs_irr}),
$$\Ga_N=\{x\in\Z^n_{\geq 0}\text{ such that } |x|=N\},$$
indexed by $N$. We consider this setting as in the beginning of $\S$ \ref{conde_} with $n=|\cS|$. We only treat product-form stationary distributions and assume they are given by a family of (product-form) probability distributions of the form
\begin{equation}\label{ass_prod}\pi_N(x)=\frac{\prod_{i\in\cS}\mu_i^{x_i}w_i(x_i)}{Z_N},\end{equation}
along sets (resp. irreducible components for CRNs) of the form (\ref{abs_irr}) where for simplicity $w_i(0)=1$ and $\mu_i> 0$. For  a fixed mass-preserving CRN $(\cG,\k)$ we denote the stationary distributions on the irreducible component with total molecule number equal to $N$ by $\pi_{\cG,N}(x)$. If we make a more general statement we stick to the notation $\pi_N(x)$. Observe also that the definitions of condensation are independent of product-form assumption of (\ref{ass_prod}).
We first check that for Poisson product-form stationary distributions we have no condensation.
One can reduce the statement to two species using the multinomial theorem, from which it is easy to deduce.
\begin{proposition} \label{no_cond}Let $(\cG^*,\k)$ be autocatalytic with only monomolecular reactions, i.e. such that the stationary distribution consists of product-form functions of Poisson type, denoted $g(m)$ in $\S$ \ref{quant_ana}. Then for $(\cG,\k)$ we have no condensation of the form (C3) (hence in any of the forms given in Definition \ref{condens_ab}).
\end{proposition}
Mass-preserving complex balanced CRNs have stationary distributions of the same product-form functions, the same result holds.
\begin{proposition} Mass-preserving complex balanced CRNs $(\cG,\k)$ have no weak condensation.
\end{proposition}
Next we introduce a generalization of \cite[Theorem 3.1]{inclusion_proc} allowing all product functions of our model. If monomolecular, threemolecular or higher reactions of Theorem \ref{prod_th} are included then product-form functions $q(m)$ can in general not be factorized as $q(m)=\mu^mw(m)$ such that \begin{equation}\label{exist_lim}\lim\limits_{m\to\infty}\frac{w(m+1)}{w(m)}= c\end{equation} (only $h$ as denoted in $\S$ \ref{quant_ana} can be manipulated such that $c=1$), see Lemma \ref{growth_lemma}, and they are not necessarily normalizable anymore, see Lemma \ref{norm_lemma}. 
\begin{remark}
Note that the limit of the quotient \eqref{exist_lim} exists for $w(m)$ if and only if the limit for $q(m)$ exists.
\end{remark}
The conditions (1) \& (2) of Theorem \ref{FirstCondensationResuslt} only require the product-function for coordinate $i^*$ to dominate the others, which gets rid of the assumption of both existence of limit \eqref{exist_lim} and normalizability. 
So for a big class of product-form function when paired with asymmetric $\mu_i$ the stochastic dynamics show a weak form of condensation as in definition \ref{condens_ab} (C3). 
\begin{theorem}\label{FirstCondensationResuslt}
Let $$\pi_N(x)=Z_N^{-1}\prod_{S_i\in\cS}^{ } \mu_i^{x_i}w_i(x_i)$$
be a family of probability measures given by product-form functions $w_i$
for $x\in\Z^n_{\geq 0},|x|=N$ and where $Z_N$ is the normalizing constant defined by:
$$Z_N=\sum_{x\in\Z_{\geq 0},|x|=N}\prod_{S_i\in\cS}^{ } \mu_i^{x_i}w_i(x_i)$$
Assume there is a $S_i^*\in\cS$ such that 
\begin{enumerate}
\item $\mu_{i^*}>\mu_j$ when $j\ne i^*$,
\item  For all $S_j\in \cS$ and all $\a>0$ there is $c_{\a,j}\in \R_{>0}$ and a $M_c\in \N$ such that for any $M>M_c$ and all $r\in \{0,\cdots M\}$ we have $$w_j(M-r)w_{i^*}(r)\leq c_{\a,j} e^{\a M}w_{i^*}(M).$$  
\end{enumerate}
Then $\pi_N$ condensates on the subset $S^*=\{S_{i^*}\}$, that is, for all $\delta \in (0,1)$
$$\lim\limits_{N \to \infty} \pi_{N}(X_{i^*}\geq (1-\delta)N)=1,$$
i.e. we have a weak form of condensation as in definition \ref{condens_ab} (C3) and we have a strong law of large numbers $\frac{X_{i^*}}{N}\to 1$ a.s. as $N\to\infty$.
\end{theorem}
\begin{proof}
Let $\delta\in(0,1)$ be arbitrary, we will show the equivalent statement that probability of the complement to this set goes to zero. 
We want to estimate
$$\pi_N(X_{i^*}\leq(1-\delta)N)=\frac{\sum_{x\in\Z_{\geq 0},|x|=N,x_{i^*}\leq (1-\delta)N}\prod_{S_i\in\cS} \mu_i^{x_i}w_i(x_i)}{Z_N}.$$
We first use the inequality $Z_N\geq  \mu_{i^*}^Nw_{i^*}(N)$ to get
$$\pi_N(X_{i^*}\leq(1-\delta)N)\leq \frac{\sum_{x\in\Z_{\geq 0},|x|=N,x_{i^*}\leq (1-\delta)N}\prod_{S_i\in\cS} \mu_i^{x_i}w_i(x_i)}{\mu_{i^*}^Nw_{i^*}(N)}.$$
We assume that $\mu_1= \max\limits_{S_j\in\cS\setminus S_{i^*}}\mu_j$. We will recursively apply the second hypothesis of Theorem \ref{FirstCondensationResuslt} $|\cS|$ times with a fixed $\a>0$ chosen such that 
\begin{equation} \label{eq_exp}e^{\a|\cS|}<(\frac{\mu_{i^*}}{\mu_1})^\delta.\end{equation}
Notice that for $N-x_{i^*}\geq \delta N$ big enough, if $x\in\Z^\cS_{\geq 0}$ is such that $\sum_{S_i\in \cS\setminus S_{i^*}}x_i=N-x_{i^*}$ then $\max_{S_i\in \cS\setminus S_{i^*}}x_i\geq \frac{N-x_{i^*}}{\mid \cS\mid}$. So we can apply $|\cS|$ times the second hypothesis in the form
$$w_j(x_j)w_{i^*}(r)\leq c_{\a,j} e^{\a (x_j+r)}w_{i^*}(x_j+r)\leq c_{\a,j} e^{\a N}w_{i^*}(x_j+r).$$
We will not write explicit dependence on the constants $c_{\a,j}$ (where $S_j\in \cS$) and just write $c$ for their assembly, since the results derived are asymptotic and hold up to multiplication by constants.
With this we derive the inequality
$$ \prod_{j\in\cS} w_j(x_j) \leq c e^{\a |\cS|N}w_{i^*}(N).$$
Applying this inequality together with 
$$\prod_{j\in\cS\setminus S_{i^*}} \mu_j^{x_j}\leq \mu_1^{N-x_{i^*}}$$
we can estimate 
$$ \prod_{j\in\cS} \mu_j^{x_j}w_j(x_j) \leq c \mu_{i^*}^{x_{i^*}}\mu_1^{N-x_{i^*}} e^{\a |\cS|N}w_{i^*}(N).$$
Then utilizing this inequality at the same time as a rough inequality
$$\mid\{ x\in \Z_{\geq 0}^{\cS\setminus S_{i^*}}\mid |x|_1=N-x_{i^*} \}\mid\leq N^{\mid \cS\mid-1}$$
for the number of integer points on the simplex we get
 $$\pi_N(X_{i^*}\leq(1-\delta)N) \leq c\frac{\sum_{x_{i^*}\leq (1-\delta)N}\mu_{i^*}^{x_{i^*}}\mu_1^{N-x_{i^*}} e^{\a |\cS|N}w_{i^*}(N) N^{|\cS|-1}}{\mu_{i^*}^Nw_{i^*}(N)}$$
$$=c\sum_{x_{i^*}\leq (1-\delta)N}(\frac{\mu_1}{\mu_{i^*}})^{N-x_{i^*}}e^{\a |\cS|N}N^{\mid \cS\mid-1}$$
We exploit that for $x_{i^*}\leq (1-\delta)N$ we have
$$(\frac{\mu_1}{\mu_{i^*}})^{N-x_{i^*}} \leq (\frac{\mu_1}{\mu_{i^*}})^{\delta N}=( (\frac{\mu_1}{\mu_{i^*}})^{\delta} )^N $$
to obtain the following inequality
$$\pi_N(X_{i^*}\leq(1-\delta)N)\leq c\sum_{m\leq (1-\delta)N}( (\frac{\mu_1}{\mu_{i^*}})^{\delta} )^N e^{\a |\cS|N}N^{\mid \cS\mid-1}.$$
Since the terms in the sum do not depend on $m$ we estimate $$|\{0\leq m\leq (1-\delta)N,m\in \Z_{\geq 0}\}|\leq N$$
to upper bound the number of terms in the sum and get
\begin{equation}\label{final}\pi_N(X_{i^*}\leq(1-\delta)N)\leq c ( (\frac{\mu_1}{\mu_{i^*}})^{\delta} )^N e^{\a |\cS|N}N^{\mid \cS\mid}=c ( (\frac{\mu_1}{\mu_{i^*}})^{\delta} e^{\a |\cS|})^N N^{\mid \cS\mid}\end{equation}
Now observe $ (\frac{\mu_1}{\mu_{i^*}})^{\delta} e^{\a |\cS|}<1$ by \ref{eq_exp} and the other factor is a polynomial in $N$ so that we conclude that this expression goes to zero for $N\to \infty$.
Since $X_{i^*}\leq N$ we use Borel-Cantelli applied to sums of $$\pi_N(X_{i^*}\leq(1-\delta)N)=\pi_N(|1-\frac{X_{i^*}}{N}|>\delta)$$ and conclude $\frac{X_{i^*}}{N}\to 1$ a.s. as $N\to\infty$. The finiteness of the series follows by combination of the direct comparison test and the ratio test for sequences applied to the final inequality term we derived in \eqref{final}.
\end{proof}
\begin{remark}\label{conv_rem} Let $(a_n)_{n\in\N}$ be a sequence of positive real numbers with
\begin{equation}\label{eq_rat}\lim\limits_{m\to\infty}\frac{a_{m+1}}{a_m}= b.\end{equation}
If $b=1$ then this implies that for all $\a>0$ there exists $c_\a$ such that for all $m\in \Z_{\geq 0}$
$$c_\a^{-1}e^{-\a m}\leq a_m\leq c_\a e^{\a m}.$$
If $0\leq b<1$ then this implies that for all $\a>0$ there exists $c_\a$ such that for all $m\in \Z_{\geq 0}$
$$ a_m\leq c_\a e^{\a m}.$$
These conclusions follow directly from \eqref{eq_rat} by limited growth (arguments e.g. as usedin the ratio test for series), also see \cite[3.2 Generalizations]{inclusion_proc}.
\end{remark}

As an application of Theorem \ref{FirstCondensationResuslt} and Remark \ref{conv_rem} we show that some asymmetric autocatalytic CRNs exhibit condensation as in definition \ref{condens_ab} (C3) if they have at least bimolecular reactions. 

\begin{corollary} \label{Corollarycondensation}
Let $(\cG^*,\k)$ be an autocatalytic CRN with highest molecularity denoted by $n^*$. Assume $n^*\geq 2$ and that there is a $S_{i^*}\in\cS$ with incoming reaction of molecularity $n_{i^*}=n^*$ such that for all other species $S_j\in\cS$ of the same molecularity $n_j=n^*$ we have
$$\la_{i^*}\be_{i^*}^{n^*}>\la_j\be_j^{n^*}$$
 Then $(\cG^*,\k)$ shows a weak form of condensation as in definition \ref{condens_ab} (C3) and we have a strong law of large numbers $\frac{X_{i^*}}{N}\to 1$ a.s. as $N\to\infty$.
\end{corollary}
\begin{proof}By assumption and Theorem \ref{prod_th} we have that 
$$\pi_{\cG,N}(x)=Z_\Ga^{-1}\prod_{S_i\in\cS}^{ }f_i(x_i)=Z_\Ga^{-1}\prod_{S_i\in\cS}^{ } \lambda_i^{x_i} p_i(x_i)$$
for $x$ in the corresponding irreducible component. It is enough to show that we can find product-form functions $\mu_i,w_i(m)$ such that the conditions of Theorem \ref{FirstCondensationResuslt} are satisfied with $\mu_i^mw_i(m)=\la_i^mp_i(m)$ for all $m\in \N$. We distinguish cases
\begin{enumerate}
\item[($n^*=2$)] Let $\la_j\be^2_j=\max\{\la_i\be^2_i|S_i\in\cS\setminus S_{i^*}\}$, and we can assume $\la_j\be^2_j\neq 0$ since otherwise the statement is trivial. Then for the species $S_k$ with $\be^2_k\neq 0$ we choose
$$\mu_k=\la_k\be^2_k,\quad w_k(m)=\frac{p_k(m)}{(\be_k^2)^m}=\frac{\Gamma(\frac{1}{\be_k^2}+m)}{m!\Gamma(\frac{1}{\be_k^2})}$$
and for species with $\be^2_k=0$ we choose a small  $\epsilon>0$ such that
$$\mu_k=\la_{i^*}\be^2_{i^*}-\epsilon,\quad w_k(m)=\frac{p_k(m)\la_k^m}{\mu_k^m}=\frac{\la_k^m}{m!(\la_{i^*}\be^2_{i^*}-\epsilon)^m}$$
Now we go through the assumptions of Theorem \ref{FirstCondensationResuslt}; (1) follows by definition. To prove (2) we first recall the asymptotic description for Gamma function following Wendel's inequality from \cite{wendel}. This gives
$$\frac{\Gamma(x+y)}{\Gamma(x)}\simeq x^y\text{ for }y\geq 0,x\to\infty.$$
Applying this to our product-form functions $w_k(m)$ of species with $\be_k^2\neq0$ gives
$$w_k(m)=\frac{\Gamma(\frac{1}{\be_k^2}+m)}{m\Gamma(m)\Gamma(\frac{1}{\be_k^2})}\simeq \frac{1}{\Gamma(\frac{1}{\be_k^2})}m^{\frac{1}{\be_k^2}-1}=c\cdot m^{\frac{1}{\be_k^2}-1}$$ 
for a constant $c>0$. In particular we have that the limit 
$$\lim\limits_{m\to\infty}\frac{w_k(m+1)}{w_k(m+1)}= b$$
exists in both cases, if $\be^2_k=0$ then $b=0$, and if $\be^2_k\neq 0$ then $b=1$.
From this and Remark \ref{conv_rem} it is easy to see that (2) is satisfied.
\item[($n^*>2$)] The same principle applies to cases with higher molecularity, condition (2) is then a special case of Lemma \ref{inequ_lem2}.
\end{enumerate}
\end{proof}
\begin{remark} Theorem \ref{FirstCondensationResuslt} also allows an interpretation along remark \ref{cond_conditioning},  with a condensation phenomena for a family of independent random variables $Y_1,\cdots ,Y_n$ with values in $\Z_{\geq 0}$,
 as in \cite[3.2 Generalizations]{inclusion_proc} but without assumption of existence of limit \eqref{exist_lim} and normalizability.
\end{remark}
Next we show that the stationary distribution asymptotically concentrates on the disjoint singleton sets $\{x\in \Ga_N|x_j=N\}\subset\Z^n_{\geq 0}$ where $S_j$ is a species with maximal product-form function $f_j=f$ (maximal as in the sense below). This confirms existence of the strongest version of condensation as in definition \ref{condens_ab} (C1) in the three- or higher molecular autocatalytic CRN. By Lemma \ref{cond_lem_1} this implies all other forms of condensation.
The result is similar to \cite[Proposition 2.1]{Bianchi16}, where they studied a reversible inclusion process whose diffusion constant decreases along irreducible components. Moreover the proof of \cite[Proposition 3.2]{Bianchi16} follows a similar strategy. 

\begin{theorem}\label{SecondCondensationResult}Let $(\cG^*,\k)$ be an autocatalytic CRN on the set of species $\cS=\{S_1,\cdots , S_n \}$ with highest molecularity denoted by $n^*$. Assume $n^*\geq 3$ and that the first $k\geq 1$ species $\{S_1,\cdots,S_k\}$ have the same product-form function $f$(determined by $\la_1(1,\be_1^2\cdots,\be_1^{n_1})$ see Definition \ref{autocatalyticCRN}, Theorem \ref{MainFormula}) with molecularity $n_1=n^*$ such that for any $S_i\in \cS\setminus \{S_1,\cdots,S_k\}$ of the same molecularity $n_i=n^*$ we have $$\la_1\be_1^{n_1}>\la_i\be_i^{n_i}.$$
Then the stationary stochastic dynamics satisfies the following for $S_j\in \{S_1,\cdots,S_k\}$
\begin{equation}\label{str_cond_form}\lim\limits_{N \to \infty} \pi_{\cG,N}(X_j= N)=\frac{1}{k}.\end{equation}
This implies condensation as in definition \ref{condens_ab} (C1) for $(\cG^*,\k)$, by Lemma \ref{cond_rem_1}.
\end{theorem}
\begin{proof}
We will repeatedly use that for $S_i\in \cS\setminus \{S_1,\cdots,S_k\}$
\begin{equation}\label{equ_as}f_i(N)=o(1)f_1(N).\end{equation}
To show (\ref{str_cond_form}) it is enough to prove that 
$$Z_N=(k+o(1))f_1(N),$$
we prove it by induction on the number of species for $Z_{N}=\sum_{x\in \Z_{\geq 0}^n, |x|_1=N}\prod_{i=1}^nf_i(x_i)$.

\begin{enumerate}
\item{$|\cS|=2$:}
We write the partition function as
$$Z_N= f_1(N)+f_2(N)+ \sum_{i=1}^{N-1}f_1(i)f_2(N-i)=f_1(N)(1+o(1))+f_2(N)$$
We are done by combining Lemma \ref{inequ_lem2} with equation \ref{equ_as} by distinguishing the cases $f=f_2$ or $\la_1\be_1^{n_1}>\la_2\be_i^{n_2}$.
\item{$|\cS|=n\to |\cS|=n+1$:} Assume $\cS\setminus S_{n+1}$ has j species with product-form function $f$. We denote for a subset of species $A\subseteq \cS$
$$Z_{N,A}:=\sum_{x\in \Z_{\geq 0}^{A}, |x|_1=N}\prod_{S_i\in A}f_i(x_i)$$ to write the partition function as follows
$$Z_N=\sum_{i=0}^{N}f_{n+1}(i)Z_{N-i,\cS\setminus S_{n+1}}
=f_{n+1}(N)+Z_{N,\cS\setminus S_{n+1}}+\sum_{i=1}^{N-1}f_{n+1}(i)Z_{N-i,\cS\setminus S_{n+1}}.$$
We apply the induction hypothesis on the $Z_{N-i,\cS\setminus S_{n+1}}$ and get
$$Z_N= f_{n+1}(N)+f(N)(j+o(1))+\sum_{i=1}^{N-1}f_{n+1}(i)f(N-i)[j+o(1)].$$
Factorizing out $f(N)$ we obtain
$$Z_N=f_{n+1}(N)+f(N)(j+o(1)+\sum_{i=1}^{N-1} \frac{f_{n+1}(i)f(N-i)}{f(N)}[j+o(1)])$$
We apply Lemma \ref{inequ_lem2} and get
$$Z_N=f_{n+1}(N)+ f(N)\big (j+o(1)\big )$$
Now if $f_{n+1}$ is also a maximal product-form function $f$, we obtain $Z_N=f(N)(j+1+o(1))$, otherwise using identity \ref{equ_as} we stay with 
$Z_N=f(N)(j+o(1))$.
\end{enumerate}

\end{proof}

\begin{remark} Both Corollary \ref{Corollarycondensation} and Theorem \ref{SecondCondensationResult} are based on the asymptotic analysis of the product-form functions in the stationary distribution. Hence the results carry over to mass preserving CRNs which have a complex balanced (hence Poisson product-form function) and an autocatalytic part( see Remark \ref{re_cx_bal}, Example \ref{ex_cx_bal} or Definition \ref{tauto_def}). The asymptotic analysis of the product-form should also be related to corresponding open CRNs.
\end{remark}

\appendix
\section{Technical results on product-form functions}
We give two lemmas providing rough estimates for the product-form functions for the three-molecular or higher case. The proofs are straightforward, the first Lemma follows by inspection of quotients of product-form functions $f_i$ of Theorem \ref{prod_th}.
\begin{lemma} \label{inequ_lem1}\hfill\break
Let $f_i=f$ be a product-form function as in $\S$ \ref{exa} with $\be_i^3>0$. Then there is a $c>0$ and a $N_0\in \N$ such that for all $N>N_0$ we have
\begin{itemize}
\item[(1)] $\frac{f(N-1)f(1)}{f(N)}\leq c \frac{1}{N-3}$.
\item[(2)] $\frac{f(N-2)f(2)}{f(N)}\leq c \frac{1}{(N-3)(N-4)}$.
\item[(3)] $\frac{f(N-3)f(3)}{f(N)}\leq c  \frac{1}{(N-3)(N-4)(N-5)}$.
\item[(4)] For $\frac{N}{2}\geq i\geq 3$ we have $f(N-i)f(i)\leq f(N-3)f(3)$.
\end{itemize}
\end{lemma}
Putting the inequalities derived in Lemma \ref{inequ_lem1} together, we can
bound the partition function of the two species case for maximal molecularity higher than two (i.e. $\be_i^3>0$), the simple proof is omitted.
\begin{lemma} \label{inequ_lem2}\hfill\break
Let $f_1,f_2$ be product-form functions as in $\S$ \ref{exa} with $n_1,n_2$ as in (4) of Definition \ref{autocatalyticCRN} $$m=n_1\geq 3, n_1\geq n_2\text{ and }\la_1\be_1^m>\la_2 \be_2^m.$$ Then we have :
$$\sum_{i=1}^{N-1}f_1(i)f_2(N-i)\leq f_1(N)o(1)$$
where $o$ is small o from Bachmann-Landau notation.
\end{lemma}

 \bibliographystyle{plain}

 \bibliography{references} 
 
\end{document}